\documentclass[11pt]{amsart}
\usepackage{amsfonts}
\usepackage{amssymb}
\usepackage{graphicx,color}

\usepackage{a4wide}

\usepackage{amsmath}

\usepackage{amsthm}

\graphicspath{{./Figs/}}
\DeclareGraphicsExtensions{.pdf,.jpg,.png,.eps,.gif}
\title{\bf Universal lower bounds for potential energy of spherical codes}
\date{\today}

\newtheorem{theorem}{Theorem}[section]

\newtheorem{corollary}[theorem]{Corollary}

\newtheorem{conjecture}[theorem]{Conjecture}

\theoremstyle{definition}
\newtheorem{defn}[theorem]{Definition}
\newtheorem{example}[theorem]{Example}
\newtheorem{remark}[theorem]{Remark}

\newcommand{\Sp}{\mathbb{S}}
\newcommand{\R}{\mathbb{R}}

\author[P. Boyvalenkov]{P. G. Boyvalenkov $^\dagger$}
\address{Institute of Mathematics and Informatics, Bulgarian Academy of Sciences,
8 G Bonchev Str.,
1113  Sofia, Bulgaria \\
and Faculty of Mathematics and Natural Sciences, South-Western University, Blagoevgrad, Bulgaria.
}
\email{peter@math.bas.bg}
\thanks{\noindent $^\dagger$ The research of this author was supported, in part, by a Bulgarian NSF contract I01/0003.}

\author[P. Dragnev]{P. D. Dragnev $^{\dagger \dagger}$}
\address{Department of Mathematical Sciences,
Indiana-Purdue University
Fort Wayne, IN 46805, USA }
\email{dragnevp@ipfw.edu}
\thanks{\noindent $^{\dagger \dagger}$ The research of this author was supported, in part, by a Simons Foundation grant no. 282207.}

\author[D. Hardin]{D. P. Hardin$^*$}
\address{Center for Constructive Approximation, Department of Mathematics, \hspace*{.1in}
Vanderbilt University,
Nashville, TN 37240, USA  }
\email{doug.hardin@vanderbilt.edu}

\author[E. Saff]{E. B. Saff$^*$}
\email{edward.b.saff@vanderbilt.edu}
\thanks{\noindent $^*$ The research of these authors was supported, in part,
by the U. S. National Science Foundation under grants  DMS-1109266 and DMS-1412428.
}
\author[M. Stoyanova]{M. M. Stoyanova$^{**}$}
\address{Faculty of Mathematics and Informatics,
Sofia University,
5 James Bourchier Blvd.,
1164 Sofia, Bulgaria}
\email{stoyanova@fmi.uni-sofia.bg}
\thanks{
\noindent $^{**}$ The research of this author was supported, in part, by the Science Foundation of Sofia University under contract 015/2014.
}
\thanks{The authors express their gratitude to Erwin Schr\"{o}dinger International Institute for providing  conducive research atmosphere during their stay when part of this manuscript was written.}

\begin{document}

\begin{abstract} We derive and investigate lower bounds for the potential energy of finite spherical
point sets (spherical codes). Our bounds are optimal in the following sense -- they cannot be improved
by employing polynomials of the same or lower degrees in the Delsarte-Yudin method. However, improvements are sometimes possible and we provide a necessary and sufficient condition for the existence of such
better bounds. All our bounds can be obtained in a unified manner that does not depend on the potential function, provided the potential is given by an absolutely monotone function of the inner product between pairs of points, and this
is the reason for us to call them universal. We also establish a criterion for a given code of dimension $n$ and cardinality $N$ not to be LP-universally optimal, e.g. we show that two codes conjectured by Ballinger et al to be universally optimal
are not LP-universally optimal.
\end{abstract}

\maketitle
\section{Introduction}

Minimal energy configurations, maximal codes, and spherical designs have wide ranging applications in various fields of science, such as crystallography, nanotechnology, material science, information theory, wireless communications, etc. In this article we shall derive lower bounds on the potential energy of such configurations via a unified method working for a large class of potential interaction functions. A fundamental connection between our lower bounds and the classical Delsarte-Goethals-Seidel bounds on spherical designs and Levenshtein's bounds on maximal codes is presented. For a fixed dimension and code cardinality the Delsarte-Goethals-Seidel bounds serve to localize the analysis and then, as illustrated in Figure 2, the zeros of the Levenshtein optimal polynomials for maximal codes determine the optimal polynomials for a large class of potentials.

Following Levenshtein's terminology (see \cite{Lev}) we call the lower bounds that we obtain universal. This choice of terms is also consistent with its use by Cohn and Kumar in their study \cite{CK} of universally optimal energy configurations, since our bounds likewise work for all absolutely monotone potential functions of the inner product. Furthermore,
our lower bounds are attained for all sharp configurations as defined in \cite{CK}.

Let $\Sp^{n-1}$ denote the unit sphere in $\R^n$.  We refer to a finite set $C \subset \Sp^{n-1}$ as a {\em spherical code} and, for a given (extended real-valued) function $h(t):[-1,1] \to [0,+\infty]$, we define the {\em $h$-energy} of a spherical code $C$ by
\begin{equation} E(C;h):=\sum_{x, y \in C, x \neq y} h(\langle x,y \rangle), \nonumber\end{equation}
where $\langle x,y \rangle$ denotes the inner product of $x$ and $y$. Note that for $x, y \in \mathbb{S}^{n-1}$ we have $|x-y|^2=2-2\langle x,y \rangle$.

A commonly arising problem is
to minimize the potential energy provided the cardinality $|C|$ of $C$ is fixed; that is, to determine
\begin{equation}
 \mathcal{E}(n,N;h):=\inf\{E(C;h):|C|=N, \, C\subset S^{n-1}\}\nonumber
 \end{equation}
the minimum possible $h$-energy of a spherical code of cardinality $N$ (see \cite{HS, SK}). Although the theorems in Section 2 hold for general potentials $h$ we will be especially concerned with functions $h(t)$ that are {\em absolutely monotone} ({\em absolutely strictly monotone}), that is $h^{(i)}(t)\geq 0$, $i=0,1,\dots$ ($h^{(i)}(t)> 0$, $i=0,1,\dots$). Some examples of absolutely monotone potentials include the {\em Riesz $\alpha$-potential} $h(t)=[2(1-t)]^{-\alpha/2}$, $\alpha>0$, and in particular the {\em Newton potential} (when $\alpha=n-2$); the {\em Gauss potential} $h(t)=e^{2t-2}$; the {\em Korevaar potential} $h(t)=(1+r^2-2rt)^{-(n-2)/2}$, $0<r<1$. Although the logarithmic potential $h(t)=-(1/2)\ln(1-t)$ is not positive on $[-1,0]$, all its derivatives are positive and the results in this article apply to this potential as well. The situation is similar for the Fejes-T\'{o}th potential $h(t)=-[2(1-t)]^{\alpha/2}$, $0<\alpha<2$, which includes the important particular case in discrete geometry of $\alpha=1$, namely of finding configurations that maximize the sum of all mutual distances.

A general  technique
(referred to here as the {\em Delsarte-Yudin method}) for obtaining lower bounds for the $h$-energy of arbitrary spherical codes
was developed by Yudin \cite{Y} using Delsarte's linear programming method \cite{D1,DGS,KL} and was further applied by Kolushov and Yudin \cite{KY}, Andreev \cite{And}, and Cohn and Kumar \cite{CK}. These bounds depend on the choice of  polynomials satisfying certain constraints.  Here we provide explicit solutions to Delsarte's linear program based upon Levenshtein's work on maximal codes \cite{Lev3} and \cite{Lev}, which allows us to establish universal lower bounds on potential energy for a large class of potential functions $h$.

In Section 2 we describe in a unified manner results from Delsarte, Goethals and Seidel \cite{DGS} and Levenshtein
\cite{Lev2,Lev3,Lev} that  are instrumental in defining our bounds. Theorems 2.3 and 2.6 explain the importance of special type quadrature rules in determining lower bounds on energy and investigation of their optimality.
Theorem~\ref{thm3.2} is one of the main results in this paper. It gives lower bounds which are optimal in the following sense -- they cannot be improved by polynomials of the same or
lower degree that satisfy the standard linear programming constraints specified in Theorem~\ref{thm1}. On the other hand, the bounds of Theorem~\ref{thm3.2}
can be further improved in some cases and Theorem~\ref{thm4.1} gives necessary and sufficient conditions
for existence of such improvements via the so-called {\em test functions}, which were first introduced and investigated for analysis of the Levenshtein bounds for maximal codes in 1996 by Boyvalenkov, Danev and Bumova \cite{BDD}. We derive a quantitative version of \cite[Theorem 5.2]{BDD} in Theorem \ref{Qposthm}, which provides a criterion for disproving that certain codes are LP-universally optimal. As an application we prove that the two codes conjectured to be universally optimal in \cite{BBCGKS}, are not LP-universally optimal, namely their universal optimality may not be established by an ad-hoc approach similar to the $600$-cell approach given in \cite{CK, CW}.

\section{Linear programming framework and $1/N$-quadrature rules}

\subsection{Gegenbauer polynomials and the Delsarte-Yudin linear programming \\framework}

For fixed dimension $n$, the Gegenbauer polynomials \cite{Sze} are
defined by  $P_0^{(n)}=1$, $P_1^{(n)}=t$ and the three-term recurrence relation
\[ (i+n-2)P_{i+1}^{(n)}(t)=(2i+n-2)tP_i^{(n)}(t)-iP_{i-1}^{(n)}(t)
                \mbox{ for } i \geq 1. \]
We note that $\{P_i^{(n)}(t)\}$ are orthogonal in $[-1,1]$ with respect to the weight $(1-t^2)^{(n-3)/2}$ and that $P_i^{(n)}(1)=1$. In standard Jacobi polynomial notation (see \cite[Chapter 4]{Sze}), we have that
\begin{equation}\label{Geg_Jacobi}
P_i^{(n)}(t)=\frac{P_i^{((n-3)/2,   (n-3)/2)}(t)}{P_i^{((n-3)/2,   (n-3)/2)}(1)}.
\end{equation}

Denote the space of real polynomials of degree at most $k$ by $\mathcal{P}_k$.  Any $f\in  \mathcal{P}_k$  can be uniquely expanded in terms of the Gegenbauer
polynomials as $f(t) = \sum_{i=0}^k f_iP_i^{(n)}(t)$. The
coefficients $f_i$    given by
\begin{equation}
f_i=\frac{\int_{-1}^1f(t)P_i^{(n)}(t)(1-t^2)^{(n-3)/2}\, dt}{\int_{-1}^1 \left[P_i^{(n)}(t)\right]^2(1-t^2)^{(n-3)/2}\, dt},\qquad i=0,1,\ldots,k,\nonumber
\end{equation}
play an important role in linear
programming theorems.

Let $\{Y_{k\ell}(x) : \ell=1,2,\ldots,r_k\}$  be an orthonormal basis of the
space $\mathrm{Harm}(k)$ of homogeneous harmonic polynomials in $n$ variables of
degree $k$ restricted to  $\Sp^{n-1}$,  where
\[r_k:=\dim \,\mathrm{Harm}(k)=\binom{n+k-3}{n-2}\frac {2k+n-2}{k}=\binom{n+k-1}{n-1}-\binom{n+k-3}{n-1}\,\]
and orthonormality is with respect to integration over the sphere utilizing $\sigma_n$, the normalized $(n-1)$-dimensional Hausdorff measure restricted to $\Sp^{n-1}$.   The functions $\{Y_{k\ell}$, $ \ell=1,2,\ldots,r_k\}$, are known as {\em spherical harmonics} of degree $k$.
 The Gegenbauer polynomials and  spherical harmonics are related through the well-known  {\em Addition Formula} (see \cite{Koor}):
\begin{equation}\label{addform}
\frac{1}{r_k}\sum_{\ell=1}^{r_k}Y_{k\ell}(x)Y_{k\ell}(y)=P_k^{(n)}(\langle x, y \rangle ),\ \ \ x,y\in \Sp^{n-1};\end{equation}
that is, the Gegenbauer polynomial $P_k^{(n)}(t)$ is, up to a normalization, the kernel for the orthogonal projection onto $\mathrm{Harm}(k)$.

If $f$ is a function integrable on $[-1,1]$ with respect to the weight function $(1-t^2)^{(n-3)/2}$ and $y$ is any fixed point on $\Sp^{n-1}$, then
the following relation (a partial case of the Funk-Hecke formula,  see \cite[Theorem 6]{Mul1966}) holds:
\begin{equation}\label {FunkHecke}
\int_{S^{n-1}}f(\langle x, y \rangle )d\sigma_n (x)=\gamma_n\int\limits_{-1}^{1}f(t)(1-t^2)^{(n-3)/2}dt,\nonumber
\end{equation}
 where
\begin {equation}\label {gamma_p}
\gamma_n:=
\frac{\Gamma
   \left(\frac{n}{2}\right)}{\sqrt{\pi } \Gamma
   \left(\frac{n-1}{2}\right)}.\nonumber
\end {equation}

If   $C=\{x_1,\ldots,x_N\}$ is a spherical code of $N$ points on $\Sp^{n-1}$, then  it follows from  \eqref{addform} that:
\begin {equation}\label {G_pos_def1}
 \sum_{i,j=1}^{N}P_k^{(n)}(\langle x_i, x_j \rangle )=\frac{1}{r_k} \sum_{\ell=1}^{r_k}\sum_{i,j=1}^{N}Y_{k \ell}(x_i)Y_{k \ell}(x_j)= \frac{1}{r_k}\sum_{\ell=1}^{r_k}\left(\sum_{i=1}^{N}Y_{k \ell}(x_i)\right)^2\geq 0.
\end {equation}
We define the $k$-th moment of $C$ by
 \begin{equation}\label{Mk0}
 M_k(C):=\sum_{i,j=1}^N P^{(n)}_k(\langle x_i , x_j\rangle ). \nonumber\end{equation}  From \eqref{G_pos_def1}, we have $M_k(C)=0$   if and only if
$\sum_{i=1 }^N Y(x_i)=0$ for all spherical harmonics $Y\in\mathrm{Harm}(k)$.  If $M_k(C)=0$ for  $1\le k\le \tau$, then  $C$  is called a
{\em spherical $\tau$-design}.  Equivalently,
$C$ is  a spherical $\tau$-design if and only if \begin{eqnarray*}
\label{defin_f.1}
 \int_{\mathbb{S}^{n-1}} p(x) d\sigma_n(x)= \frac{1}{|C|} \sum_{x \in C} p(x)
\end{eqnarray*}
($\sigma_n $ is the normalized $(n-1)$-dimensional Hausdorff measure) holds for all polynomials $p(x) = p(x_1,x_2,\ldots,x_n)$ of degree at most $\tau$.  The set
\begin{equation}\label{I_C}
\mathcal{I}(C):=\{k\in \mathbb{N}\colon M_k(C)=0\},
\end{equation}
 is called the {\em index set} of $C$.  Hence, $C$ is a spherical $\tau$-design if and only if  $\{1, 2, \ldots, \tau\}\subset \mathcal{I}(C)$.  \\

Suppose $f:[-1,1]\to \R$   is of the form
\begin{equation}\label{fform}
f(t)=\sum_{k=0}^\infty f_k P_k^{(n)}(t),\qquad     f_k\ge 0   \text{ for all } k\ge 1,
\end{equation}  where we remark that $f(1)=\sum_{k=0}^\infty f_k<\infty$.  Since $|P_k^{(n)}(t)|\le 1$, it follows that the right-hand side of \eqref{fform} converges uniformly on $[-1,1]$.   We then obtain the following relations which form the
basis for many packing and energy bounds for spherical codes $C=\{x_i\}_{i=1}^N $ of cardinality $N$ (see \cite{BHS,CK, KL, Y}):
\begin{equation}\label{glowbound}
\begin{split}
E(C;f)&=  \sum_{i,j=1}^N f(\langle x_i, x_j \rangle )-f(1)N\\ &=\sum_{k=0}^\infty f_k \sum_{i,j=1}^{N}P_k^{(n)}(\langle x_i, x_j \rangle )-f(1)N\\ &=\sum_{k=0}^\infty f_kM_k(C)-f(1)N\\ &\ge f_0N^2-f(1)N.
\end{split}
\end{equation}

Since $M_k(C)=0$ for $k=1,\ldots , \tau$ when $C$ is a $\tau$-design,   the following result immediately follows from \eqref{glowbound}.

\begin{theorem}[Delsarte, Goethals, Seidel \cite{DGS}]\label{DGS1}
Suppose $C$ is a spherical $\tau$-design on $\Sp^{n-1}$ and   $f(t)$ is a polynomial of degree at most $\tau$ such that
$f(t)\ge 0$ on $[-1,1]$ and $f_0=\gamma_n \int_{-1}^1 f(t)(1-t^2)^{(n-3)/2}\, dt>0$.  Then
\begin{equation}\label{Db}
|C|\ge \frac{f(1)}{f_0}.
\end{equation}
\end{theorem}

Maximizing the right hand side of \eqref{Db} over polynomials satisfying the above hypotheses, Delsarte, Goethals, and Seidel \cite{DGS} obtain
a lower bound on
$$B(n,\tau):=\min\{|C|: C \subset \mathbb{S}^{n-1} \mbox{ is a spherical $\tau$-design}\}$$
Specifically, they show
\begin{equation}
\label{DGS-bound}
B(n,\tau) \geq D(n,\tau) := \left\{
\begin{array}{ll}
\displaystyle{ 2{n+k-2 \choose n-1}},& \mbox{if $\tau=2k-1$,} \\
 & \\
  \displaystyle{{n+k-1 \choose n-1}+{n+k-2 \choose n-1}}, & \mbox{if   $\tau=2k$}.
\end{array}
  \right.
\end{equation}
We refer to $D(n,\tau)$ as the {\em Delsarte-Goethals-Seidel  bound} for spherical $\tau$-designs.

\medskip

Another application of \eqref{glowbound} is Yudin's lower bound on energy.

\begin{theorem}[Yudin \cite{Y}]\label {thm1}
Suppose $f:[-1,1]\to \R$   is of the form
\eqref{fform}
with $f_k\ge 0$ for all $k\ge 1$.  Then, for $N\ge 2$
\begin{equation}\label{glowbound2}
 \mathcal{E}(n,N;f)\ge  f_0 N^2-f(1) N.\nonumber
\end{equation}
Consequently, if  $h:[-1,1]\to [0,\infty]$ satisfies $h(t)\ge f(t)$, $t\in [-1,1]$, we have
\begin{equation}\label{EKlowbound}
 \mathcal{E}(n,N;h)\ge  f_0N^2-f(1)N.
\end{equation}
Furthermore, $C$ is an optimal (energy minimizing)  code for $h$ and equality holds in \eqref{EKlowbound} if and only if  both of the following conditions hold:
\begin{itemize}
\item[{\rm (a)}] $f(t)=h(t)$  for all $t\in \{ \langle x, y \rangle  :  x\neq y, \ x,y \in C \};$
 \item[{\rm (b)}]  for all $k\ge 1$,   either $f_k=0$ or
$M_k (C)=0$.
\end{itemize}
\end{theorem}

For a given $h:[-1,1]\to [0,\infty]$, we denote by $A_{n,h}$ the set of functions $f\le h$ satisfying the conditions \eqref{fform}. Recall that for such $f$, the coefficient sequence $(f_0,f_1,\ldots )\in \ell_1$.   The problem of maximizing the lower bound $f_0 N^2- f(1)N$ arising in Theorem~\ref{thm1} can then be expressed in terms of an infinite linear program:

 \begin{equation}\label{lp_problem}
\begin{split}
\text{maximize} \quad  &F(f_0,f_1,\ldots ) := N\left(f_0(N-1)-\sum_{k= 1}^\infty f_k \right), \\
 \text{subject to }   &   \sum_{k=0}^\infty f_kP_k^{(n)}(t)\le h(t), t\in [-1,1]\, \text{ and }      f_k\ge 0, \text{ for all } k\ge 1.
 \end{split}
 \end{equation}
\medskip

In the following we shall consider   the above linear program   restricted  to a subspace $\Lambda$ (usually finite-dimensional)  of the linear space $C([-1,1])$ of real-valued functions continuous  on $[-1,1]$.  For such a $\Lambda$, we define
\begin{equation}\label{Wdef} \mathcal{W}(n,N,\Lambda;h):= \sup_{f \in \Lambda \cap  A_{n,h}} N^2(f_0 -f(1)/N) . \end{equation}
In general, it can be a difficult problem to find the value of $\mathcal{W}(n,N,\Lambda;h)$.  We consider sufficient conditions that allow us to
solve for  $\mathcal{W}(n,N,\Lambda;h)$. In particular, we explicitly find the solutions of the truncated linear program \eqref{lp_problem} and thus find \eqref{Wdef} when $\Lambda=\mathcal{P}_k$, for all $k\leq \tau(n,N)$, for some $\tau(n,N)$ (as defined in equation \eqref{tauNn} below). In the particular case when $m=\tau(n,N)$ we derive the universal lower bound (ULB) for potential energy of spherical codes.

\subsection{$1/N$-Quadrature rules and  lower bounds for energy}
  We refer to a finite sequence
of ordered pairs $\{(\alpha_i, \rho_i)\}_{i=1}^{k}$ as a {\em $1/N$-quadrature rule}  if
  $-1 \leq \alpha_1 < \alpha_2 <
\cdots <\alpha_{k} < 1$,
 and
$\rho_i>0$ for $i=1,2,\ldots,k$,
 and say that $\{(\alpha_i, \rho_i)\}_{i=1}^{k}$ is {\em exact} for a subspace  $\Lambda\subset C([-1,1])$  if
\begin{equation}
\label{defin_f0.1}
f_0:=\gamma_n\int_{-1}^1f(t)(1-t^2)^{(n-3)/2}dt= \frac{f(1)}{N}+ \sum_{i=1}^{k} \rho_i f(\alpha_i),
\end{equation}
for all $f\in \Lambda$.

\begin{theorem}\label{THM_subspace}
Let $\{(\alpha_i, \rho_i)\}_{i=1}^{k}$ be a $1/N$-quadrature rule that is exact for a subspace $\Lambda\subset C([-1,1])$.
\begin{enumerate}
\item[(a)] If $f\in \Lambda\cap A_{n,h}$, then
\begin{equation}\label{Quad1}
\mathcal{E}(n,N;h)\ge  N^2\sum_{i=1}^{k} \rho_i f(\alpha_i).\nonumber
\end{equation}
\item[(b)] We have \begin{equation}\label{Wineq1}\mathcal{W}(n,N,\Lambda;h)\le N^2\sum_{i=1}^{k} \rho_i h(\alpha_i).\end{equation}
\end{enumerate}
If there is some $f\in \Lambda \cap  A_{n,h}$ such that $f(\alpha_i)=h(\alpha_i)$ for $i=1,\ldots, k$, then
equality holds in \eqref{Wineq1} which yields the   universal lower bound
\begin{equation}\label{ULB1}
\mathcal{E}(n,N;h)\ge  N^2\sum_{i=1}^{k} \rho_i h(\alpha_i).
\end{equation}
\end{theorem}

\begin{proof}
If $f\in \Lambda$, then \eqref{defin_f0.1} holds and so, from Theorem~\ref{thm1}, we obtain
$$
\mathcal{E}(n,N;h)\ge N^2(f_0- {f(1)}/{N})=N^2\sum_{i=1}^{k} \rho_i f(\alpha_i),
$$
showing that (a) holds.

For (b),  using \eqref{defin_f0.1}, we obtain
\begin{equation}
\begin{split}\mathcal{W}(n,N,\Lambda;h)&= \sup_{f \in \Lambda \cap  A_{n,h}} N^2(f_0 -f(1)/N) \\ &=\sup_{f \in \Lambda \cap  A_{n,h}}N^2\sum_{i=1}^{k} \rho_i f(\alpha_i)\le   N^2\sum_{i=1}^{k} \rho_i h(\alpha_i).
\end{split}\nonumber
\end{equation}
Clearly equality holds  if  there is some $f\in \Lambda \cap  A_{n,h}$ such that $f(\alpha_i)=h(\alpha_i)$ for $i=1,\ldots, k$.
\end{proof}

As we next describe, a spherical code $C=\{x_1,\ldots, x_N\}\subset \Sp^{n-1}$  provides a quadrature rule that is exact on the subspace $$\Lambda_C:= \left\{f(t)=f_0+\sum_{l\in \mathcal{I}(C)} f_lP_l^{(n)}(t)\colon  \sum_{l\in \mathcal{I}(C)} |f_l|<\infty \right\},
$$
with $I(C)$ as defined  in \eqref{I_C}.
Let
$$
\left\{\langle x_i, x_j\rangle\colon x_i\neq x_j \in C\right\}=:\{-1\le \alpha_1<\alpha_2<\cdots<\alpha_{k}<1\},
$$
and let $\{q_l\}$ denote the inner product distribution; i.e., $$q_l:=\frac {\big|\{(i,j) \colon \langle x_i,x_j\rangle=\alpha_l\}\big|}{N^2}, \qquad l=1,\ldots, k.$$
If $f\in \Lambda_C$, then $f_l=0$ for all $l\not \in\mathcal{I}(C)$ (unless $l=0$) and   equality holds in  \eqref{glowbound}.  Hence, for such $f$, we obtain
\begin{equation}\label{C_quadrature}
 f_0=\frac{1}{N^2}\big(E(C;f)+Nf(1)\big)=  \frac{f(1)}{N}+\sum_{l=1}^{k}q_l f(\alpha_l),
\end{equation}
that is,  $\{(\alpha_l, q_l)\}_{l=1}^{k}$  is a $1/N$-quadrature rule  exact for   $\Lambda_C$.

\noindent
\begin{example}\label{600-cell}
As an example we consider the 600-cell  $C$ consisting of 120 points  in  $\Sp^3$.  Each $x\in C$ has 12 nearest neighbors forming an icosahedron (the Voronoi cells are dodecahedra) and there are $8$  inner products $-1=\alpha_1<\alpha_2<\cdots< \alpha_8<1$ between distinct points in $C$.
If $f(t)\le h(t)$ on $[-1,1]$ and $f(\alpha_k)=h(\alpha_k)$ and  for all $\alpha_k>-1$, then we must also have $f'(\alpha_k)=h'(\alpha_k)$, resulting in $2\cdot 7+1=15$ interpolation conditions.   If $C$ were a 14-design, then this would suggest  we search for $f\in A_{4,h} \cap  \Lambda$ with $\Lambda=\mathcal{P}_{14}$.  However,  $C$ is only an 11-design (i.e., $M_{12}(C)\neq 0$), although $M_{13}(C)=\cdots = M_{19}(C)=0$, so $C$ is almost a 19-design.  This suggests we choose $\Lambda$ to be a 15-dimensional subspace of  $\mathcal{P}_{19}\cap \{P_{12}^{(4)}\}^\perp$.     In fact,  Cohn and Kumar \cite[Section 7]{CK} show that for any  absolutely monotone potential $h$ on $[-1,1]$, there is a unique $f\in A_{n,h}\cap \Lambda$ for $\Lambda:=\{ f\in \mathcal{P}_{17}\colon f_{11}=f_{12}=f_{13}=0\}$ that proves the optimality of $C$.
\end{example}

\begin{example} \label{sharp_conf}
Another example is provided by the so-called {\em sharp configurations} \cite{CK}, namely configurations with $k$ distinct inner products that are spherical designs of strength $2k-1$. In this case $\Lambda=\mathcal{P}_{2k-1}$ and the existence of the $1/N$-quadrature is provided by the configuration quadrature \eqref{C_quadrature} and the design property. We shall return to this example in the Remark \ref{Rmk3.3} following Theorem \ref{thm3.2}.
\end{example}

The two examples above cover all currently known universally optimal configurations. The next theorem provides sufficient conditions for optimality of \eqref{ULB1} even in a larger subspace.

\begin{theorem}\label{THM_subspace_improve}
 Let $\{(\alpha_i, \rho_i)\}_{i=1}^{k}$ be a $1/N$-quadrature rule that is exact for a  subspace $\Lambda\subset C([-1,1])$ and such that
 equality holds in \eqref{Wineq1}.  Suppose $\Lambda'=\Lambda \bigoplus \text{span }\{P_j^{(n)} \colon  j\in I\}$ for some   index set $I\subset {\mathbb N}$.
If $Q_j^{(n)}:= \frac{1}{N}+\sum_{i=1}^{k}\rho_i P_j^{(n)}(\alpha_i)   \ge 0 $ for $j\in I$, then
\[
\mathcal{W}(n,N,\Lambda';h)=\mathcal{W}(n,N,\Lambda;h)=N^2\sum_{i=1}^{k}\rho_i h(\alpha_i).
\]
\end{theorem}
\begin{proof}
  Suppose $f(t) \in A_{n,h}\cap\Lambda'$.  Then  we may write the decomposition of $f$ as
\begin{equation}
\label{n1}
f(t)= g(t)+\sum_{j\in I}  f_j P_j^{(n)}(t),\nonumber
\end{equation}
for some $g\in \Lambda$ and $f_j\ge 0$, for $j\in I$.  Note that $f_0=g_0$, since $0\not \in I$. Furthermore, since the quadrature rule  $\{(\alpha_i, \rho_i)\}_{i=1}^{k}$ is exact for $g\in \Lambda$, we have
\begin{eqnarray*}
f_0- f(1)N^{-1} &=& g_0 - f(1)N^{-1} = \frac{g(1)}{N}+\sum_{i=1}^{k} \rho_i g(\alpha_i) -\left(g(1)+\sum_{j\in I}  f_j \right)N^{-1} \\
    &=& \sum_{i=1}^{k} \rho_i \left(f(\alpha_i)-\sum_{j\in I}  f_j P_j^{(n)}(\alpha_i)\right)-\left(\sum_{j\in I}  f_j \right)N^{-1}\\
    &=&\sum_{i=1}^{k} \rho_i  f(\alpha_i)-\sum_{j\in I}  f_j \left(\frac{1}{N}+ \sum_{i=1}^{k}\rho_i P_j^{(n)}(\alpha_i) \right)\\
    &=& \sum_{i=1}^{k} \rho_i  f(\alpha_i)-\sum_{j\in I}f_jQ_j^{(n)} \le \sum_{i=1}^{k} \rho_i  h(\alpha_i)=\frac{1}{N^2}\mathcal{W}(n,N,\Lambda;h),
\end{eqnarray*}
 where, for the last inequality, we used $f(t) \in A_{n,h}$ and $Q_j^{(n)} \geq 0$.
 \end{proof}

\subsection{Levenshtein bounds for spherical codes}

Let
$$A(n,s):=\max\{|C| \colon C \subset \mathbb{S}^{n-1}, \langle x,y \rangle \leq s, \,   x\neq y \in C\}$$ denote the maximal possible cardinality of a spherical code on $\mathbb{S}^{n-1}$ of prescribed maximal
inner product $s$.

For $a,b \in \{0,1\}$ and $i \geq 1$, let $t_i^{a,b}$   denote the greatest zero of the adjacent
Jacobi polynomial $P_i^{(a+\frac{n-3}{2},b+\frac{n-3}{2})}(t)$ and also define $t_0^{1,1}=-1$.   For $\tau\in \mathbb{N}$, let   $\mathcal{I}_\tau$ denote the interval
\begin{eqnarray*}
  \mathcal{I}_\tau :=
\left\{
\begin{array}{ll}
    \left [ t_{k-1}^{1,1},t_k^{1,0} \right ], & \mbox{if } \tau=2k-1, \\ &
    \ \\
    \left [ t_k^{1,0},t_k^{1,1} \right ],      & \mbox{if } \tau=2k, \\
  \end{array}\right.
\end{eqnarray*}
The collection of intervals is well defined from the interlacing properties $ t_{k-1}^{1,1}<t_k^{1,0}<t_k^{1,1}$, see \cite[Lemmas 5.29, 5.30]{Lev}. Note also that it partitions  $\mathcal{I}=[-1,1)$ into countably many subintervals with non-overlapping interiors.

For every $s \in \mathcal{I}_\tau$, using linear programming
bounds for special polynomials $f_\tau^{(n,s)}(t)$ of degree $\tau$ (see \cite[Equations (5.81) and (5.82)]{Lev}), Levenshtein proved that (see \cite[Equation (6.12)]{Lev})
\begin{equation}
\label{L_bnd}
 A(n,s) \leq
\left\{
\begin{array}{ll}
    L_{2k-1}(n,s) = {k+n-3 \choose k-1}
         \big[ \frac{2k+n-3}{n-1} -
          \frac{P_{k-1}^{(n)}(s)-P_k^{(n)}(s)}{(1-s)P_k^{(n)}(s)}
         \big] ,&\mbox{if }s\in \mathcal{I}_{2k-1}\cr
& \\
    L_{2k}(n,s) = {k+n-2 \choose k}
        \big[ \frac{2k+n-1}{n-1} -
           \frac{(1+s)( P_k^{(n)}(s)-P_{k+1}^{(n)}(s))}
    {(1-s)(P_k^{(n)}(s)+P_{k+1}^{(n)}(s))} \big] ,&\mbox{if } s\in \mathcal{I}_{2k}.\cr
   \end{array}\right.
\end{equation}
For every fixed dimension $n$ each bound $L_\tau (n,s)$ is smooth with respect to $s$. The function
\[ L(n,s) =
\left\{
\begin{array}{ll}
    L_{2k-1}(n,s),& \mbox{ if } s \in \mathcal{I}_{2k-1}, \cr
\\
   L_{2k}(n,s), & \mbox{ if } s \in \mathcal{I}_{2k} \cr
    \end{array}\right.
\]
is continuous in $s$. The connection between the Delsarte-Goethals-Seidel bound (\ref{DGS-bound}) and the
Levenshtein bounds (\ref{L_bnd}) is given by the equalities
\begin{equation}\label{L-DGS1}
\begin{split}
L_{2k-2}(n,t_{k-1}^{1,1})&=
L_{2k-1}(n,t_{k-1}^{1,1}) = D(n,2k-1),\\
  L_{2k-1}(n,t_k^{1,0})&=
L_{2k}(n,t_k^{1,0}) = D(n,2k).
\end{split}
\end{equation}
and the ends of the intervals $\mathcal{I}_\tau$.

\begin{figure}[ht]
\begin{center}
\vspace{1mm}
\includegraphics[scale=.34]{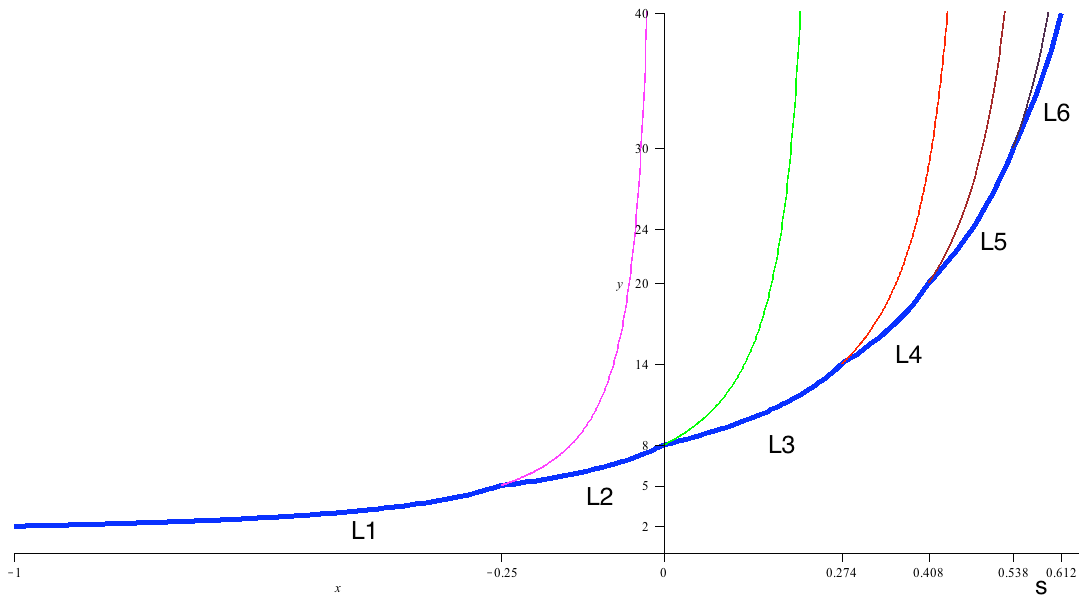} \\
\caption{The Levenshtein function $L(4,s)$ on $\mathcal{I}_k$, $1\leq k\leq 6$. }
\end{center}
\end{figure}

\subsection{Levenshtein's $1/N$-quadrature rule}

Levenshtein's method for obtaining his bounds on cardinality of maximal spherical codes utilizes orthogonal polynomials theory and Gauss-type quadrature rules that we now briefly review. The location of the cardinality $N$ relative to the Delsarte-Goethals-Seidel numbers $D(n,\tau)$ is an important step in determining our universal lower bounds. From the properties of the bounds $D(n,\tau)$ and $L_\tau (n,s)$
(see (\ref{DGS-bound}), \eqref{L-DGS1}) we derive that for every fixed dimension $n$ and cardinality $N$ there is unique
\begin{equation}
\label{tauNn}
\tau:=\tau(n,N) \quad \text{such that} \quad N \in (D(n,\tau),D(n,\tau+1)].
\end{equation}
For the so found $\tau$ define $k:= \left\lceil \frac{\tau+1}{2} \right\rceil$ and let $\alpha_k=s$ be the unique solution of
\begin{equation} \label{Lns_eq}
N=L_{\tau}(n,s), \quad s \in {\mathcal I}_\tau. \end{equation}
Then as described by Levenshtein in \cite[Section 5]{Lev} (see also \cite{Lev3,BBD}) there exist
uniquely determined quadrature nodes and nonnegative weights
\begin{equation} \label{quadrature_nodes}
-1 \leq \alpha_1 <
\cdots <\alpha_k < 1,\quad \rho_1,\ldots,\rho_k \in \mathbb{R}^+,\quad i=1,\ldots,k
\end{equation}
such that the Radau/Lobato $1/N$-quadrature (see \cite{D}, \cite{BBMQ}) holds
\begin{equation}
\label{defin_qf}
f_0= \frac{f(1)}{N}+\sum_{i=1}^{k} \rho_i f(\alpha_i), \ \ \mbox{ for all}\  f\in \mathcal{P}_{\tau}.
\end{equation}
When $\tau=2k-2$ is even, then $\alpha_1=-1$ and \eqref{defin_qf} is Lobato quadrature.
The numbers $\alpha_i$, $i=2,\ldots,k$, are the roots of the equation
\begin{equation}
P_{k-1}(t)P_{k-2}(\alpha_k) - P_{k-1}(\alpha_{k})P_{k-2}(t)=0,\label{op_eq1}
\end{equation}
where $P_i(t)=P_i^{(\frac{n-1}{2},\frac{n-1}{2})}(t)$.
When $\tau=2k-1$ is odd, then $\alpha_1>-1$ and \eqref{defin_qf} becomes Radau quadrature.
The numbers $\alpha_i$, $i=1\ldots,k$, are the roots of the equation
\begin{equation}
P_k(t)P_{k-1}(\alpha_k) - P_k(\alpha_{k})P_{k-1}(t)=0,\label{op_eq}
\end{equation}
where $P_i(t)=P_i^{(\frac{n-1}{2},\frac{n-3}{2})}(t)$. In fact, $\{\alpha_i\}$
are roots of the Levenshtein's polynomials $f_{\tau}^{(n,\alpha_{k})}(t)$ (see \cite[Equations (5.81) and (5.82)]{Lev}).

The dynamical behavior of the quadrature nodes $\{\alpha_i\} $ is the following.
When $N\in (D(n,2k-2), D(n,2k-1))$ then  $\alpha_1 =-1$ and the quadrature
\eqref{defin_qf} is Lobato. The solution $\alpha_k $ of \eqref{Lns_eq} belongs
to the interval $ (t_{k-1}^{1,0},t_{k-1}^{1,1}) $ and  all $\{\alpha_i\}_{i=2}^{k}$
strictly increase with $N$. We have that
\[ 1=|\alpha_1|>|\alpha_2|>|\alpha_k|>|\alpha_3|>|\alpha_{k-1} |>\cdots . \]
At the transition point $N=D(n,2k-1)$, $\alpha_1 =-1$ and $\alpha_k =t_{k-1}^{1,1} $. The equation \eqref{op_eq1} becomes $P_{k-1}^{(n+2} (t)=0$, which implies that
\[ 1=|\alpha_1|>|\alpha_2|=|\alpha_k|>|\alpha_3|=|\alpha_{k-1}| >\cdots . \]
As $N$ increases from $D(n,2k-1)$ to $D(n,2k)$, $\alpha_k$
strictly increases from $t_{k-1}^{1,1}$ to $t_k^{1,0}$, as do the rest of the nodes
$\{\alpha_i\}_{i=1}^{k-1}$. In particular, $\alpha_1>-1$  and \eqref{defin_qf} defines Radau quadrature and
\[ 1>|\alpha_1|>|\alpha_k|>|\alpha_2|=|\alpha_{k-1} |>\cdots . \]
More details on the nodes $\{\alpha_i\} $ can be found in \cite[Appendix]{BDL}, \cite[Corollary 3.9]{BD}, and \cite[Section 2.6]{BoumovaPhD}.

\section{Universal lower bounds}

\subsection{Optimal polynomials for lower bounds}

The optimal polynomials of degrees one and two to be applied in Theorem~\ref{thm1} can be found
by direct computations and manipulations with the corresponding derivatives. These polynomials suggest a
general form of polynomials which are optimal in the following sense -- they give lower bounds
which cannot be improved by utilizing other polynomials of the same or lower degree in Theorem~\ref{thm1}.

Our choice of polynomials for Theorem~\ref{thm1} can be viewed as extension of the ideas of Levenshtein \cite{Lev3,Lev}
who uses suitable quadrature formulas (Subsection 2.4) to explain the bounds (\ref{L_bnd}) and their
optimality in the same sense as above. This similarity should not seem unusual -- the maximal code problem
is infinite version of the Riesz energy problem. In fact, Cohn and Kumar \cite{CK} use similar idea to
deal with the universally optimal configurations. Thus, our paper can be viewed as natural extension of
the works \cite{Lev3,Lev,CK}.
Recall that given a fixed dimension $n$ and a code cardinality $N$ we can associate $\tau=\tau(n,N)$ and $ s \in  \mathcal{I}_\tau$ such that $L_\tau (n,s)=N$ (see \eqref{tauNn} and \eqref{Lns_eq}). Depending on the parity of $\tau$ we distinguish two cases:

Case (i):  $\tau=2k-2$ and $\alpha_k=s\in \left( t_{k-1}^{1,0},t_{k-1}^{1,1}\right]$. Then $f(t):=f_{\tau(n,N)}^h (t)$ is the Hermite interpolation polynomial of degree $2k-2$ defined by (recall that $\alpha_1=-1$ in this case)
\begin{equation} \label{Hermite_beta} f(-1)=h(-1), \ f(\alpha_i)=h(\alpha_i), \ f^\prime(\alpha_i)=h^\prime(\alpha_i), \ i=2,\ldots,k. \end{equation}

Case (ii): $\tau=2k-1$ and $\alpha_k=s\in \left( t_{k-1}^{1,1},t_k^{1,0}\right]$. Then $f(t):=f_{\tau(n,N)}^h (t)$ is the Hermite interpolation polynomial of degree $2k-1$ defined by
\begin{equation} \label{Hermite_alpha} f(\alpha_i)=h(\alpha_i), \ f^\prime(\alpha_i)=h^\prime(\alpha_i), \ i=1,2,\ldots,k; \end{equation}

In the notation of Cohn-Kumar's paper \cite[p. 110]{CK}, our polynomials are
\begin{equation}\label{OptPoly_Lev}
 f_{\tau(n,N)}^h (t)=H(h; (t-s)f_{\tau}^{(n,s)} (t)).\end{equation}

\subsection{Main theorem}
The equations \eqref{Hermite_beta} and \eqref{Hermite_alpha} define a Hermite's interpolation problem for $f(t)$ to intersect and touch the
graph of the potential function $h(t)$ (see \cite[Theorems 2 and 3]{KY}, \cite[Section 5]{CK}). This implies as in  \cite[Sections 3 and 5]{CK} that $f\in A_{n,h}$ and
we could use $f(t)$ for bounding ${\mathcal E}(n,N;h)$ from below. Observe that the nodes \eqref{quadrature_nodes} are independent of the potential function $h$, hence we call our bound on ${\mathcal E}(n,N;h)$ {\em a universal lower bound} (ULB).

Next, we state our main theorem. We note that here is the first time when we impose the condition that the potential function $h(t)$ is absolutely monotone and that none of the preceding results have required this property.

\begin{theorem}\label{thm3.2} Let $n$, $N$  be fixed and $h(t)$ be an absolute monotone potential. Suppose that $\tau=\tau(n,N)$ is as in \eqref{tauNn}, and choose $k= \left\lceil \frac{\tau+1}{2} \right\rceil$. Associate the quadrature nodes and weights  $\alpha_i$ and $\rho_i$, $i=1,\ldots,k$, as in \eqref{defin_qf}.
Then
\begin{equation}
\label{bound_odd}
 {\mathcal E}(n,N;h) \geq R_{\tau}(n,N;h):=N^2\sum_{i=1}^{k} \rho_i h(\alpha_i) .
 \end{equation}
Moreover, the polynomials defined by (i), respectively by (ii), provide the unique optimal solution of the linear program \eqref{Wdef} for the subspace $\Lambda=\mathcal{P}_\tau$ and consequently,
\begin{equation}\label{LP_optimality}
\mathcal{W}(n,N, \mathcal{P}_\tau;h)=R_{\tau}(n,N;h).
\end{equation}
\end{theorem}

\begin{remark} The optimality of the Hermite interpolants \eqref{OptPoly_Lev} is analogous to the optimality of the Levenshtein polynomials $f_\tau^{(n,s)} (t)$ (proved first by Sidelnikov \cite{S}), and emphasizes the universality of our bound.
\end{remark}

\begin{remark} \label{Rmk3.3} As noted in Example \ref{sharp_conf}, the sharp configurations (see \cite{CK}) define $1/N$-quadrature. Moreover, the $k$ inner products coincide with $\{\alpha_i\}$. Consequently, the bounds \eqref{bound_odd} are attained by all sharp configurations.
\end{remark}

\begin{proof}[Proof of Theorem 3.1] We first consider the odd case (ii), that is $\tau=2k-1$. The conditions in (ii) define Hermite interpolation at the points
$\alpha_i$, $i=1,2,\ldots,k$, and give a unique polynomial $f$ of degree $2k-1$ with positive
leading coefficient. The absolute monotonicity of $h(t)$ implies that $f(t)\leq h(t)$.

Next we derive that $f$ satisfies the condition \eqref{fform} as well. From \eqref{op_eq} we have that the quadrature nodes $\{\alpha_1,\dots,\alpha_{k}\}$ are zeros of the polynomial $P_k (t)+cP_{k-1}(t)$, where $\{P_i\}$ are the Jacobi orthogonal polynomials $\{P_i^{(\frac{n-1}{2},\frac{n-3}{2})}\}$. From the interlacing properties of the orthogonal polynomials we obtain that the constant $c=-P_k (s)/P_{k-1}(s)$ is non-negative. Indeed, the largest roots of the Jacobi polynomials  $t_{k-1}^{1,0}$ of $P_{k-1}$ satisfy $ t_{k-1}^{1,0}<t_{k-1}^{1,1}$ (see \cite{Lev3}). Since the last but largest root of $P_k$ is smaller than $t_{k-1}^{1,0}$ (by the interlacing property), we obtain that the ratio $P_k (t)/P_{k-1}(t)$ doesn't change sign in $[t_{k-1}^{1,1},t_k^{1,0})$. Moreover,  from \cite[Lemma 3.1.3 (a)]{BoumovaPhD}  (see also \cite[Lemma 1.5.8]{Boy_dis})

\[\displaystyle{-\frac{P_k (t_{k-1}^{1,1})}{P_{k-1}(t_{k-1}^{1,1})}=\frac{n+2k-3}{n+2k-1}}>0,\]
hence $c\geq 0$. Utilizing the approach of \cite[Sections 3 and 5]{CK} we conclude that the Hermite interpolant $f$ has non-negative Gegenbauer expansion. Therefore, $f \in A_{n,h}$.

We now use (\ref{defin_qf}) to derive the universal bound of $f$. We have
\[ f_0= \frac{f(1)}{N}+ \sum_{i=1}^{k} \rho_i
f(\alpha_i) \iff N(f_0N-f(1))=N^2\sum_{i=1}^{k} \rho_i f(\alpha_i)=N^2\sum_{i=1}^{k} \rho_i h(\alpha_i), \]
which means that
$ {\mathcal E}(n,N;h) \geq N^2\sum_{i=1}^{k} \rho_i h(\alpha_i)=R_{2k-1}(n,N;h)$.

Furthermore, for any polynomial $u =\sum_{i=0}^{2k-1} u_i P_i^{(n)}(t) \in A_{n,h}$ of degree at most $2k-1$
we have
\begin{equation} \label{Main_Ineq} N(f_0N-f(1))=N^2\sum_{i=1}^{k} \rho_i h(\alpha_i) \geq N^2\sum_{i=1}^{k} \rho_i u(\alpha_i)=N(u_0N-u(1)), \end{equation}
i.e. $N(u_0N-u(1)) \leq R_{2k-1}(n,N;h)$ and $u(t)$ does not improve (\ref{bound_odd}).

Should equality hold in \eqref{Main_Ineq} for some $u \in A_{n,h}\cap \mathcal{P}_{2k-1}$, we observe that $u(\alpha_i )=h(\alpha_i)$ for $i=1,2,\dots, k$. Additionally, the condition $u(t)\leq h(t)$ implies that $u^\prime (\alpha_i)=h^\prime (\alpha_i)$ for all $\alpha_i\in(-1,1)$. Hence, $u$ satisfies the Hermite interpolation data \eqref{Hermite_alpha}, and by the uniqueness of the Hermite interpolant, $u\equiv f$. Therefore,
$f$ is the unique optimal solution to the linear programming problem \eqref{lp_problem} in the class $A_{n,h}\cap \mathcal{P}_{2k-1}$ and \eqref{LP_optimality} holds.

In the even case (i) we proceed analogously, where we only modify the proof of the non-negativity of the Gegenbauer expansion. In this case we utilize \cite[Lemma 10]{CW}.
\end{proof}

\begin{figure}[ht]\label{K2}
\begin{center}
\vspace{1mm}
\includegraphics[scale=.53]{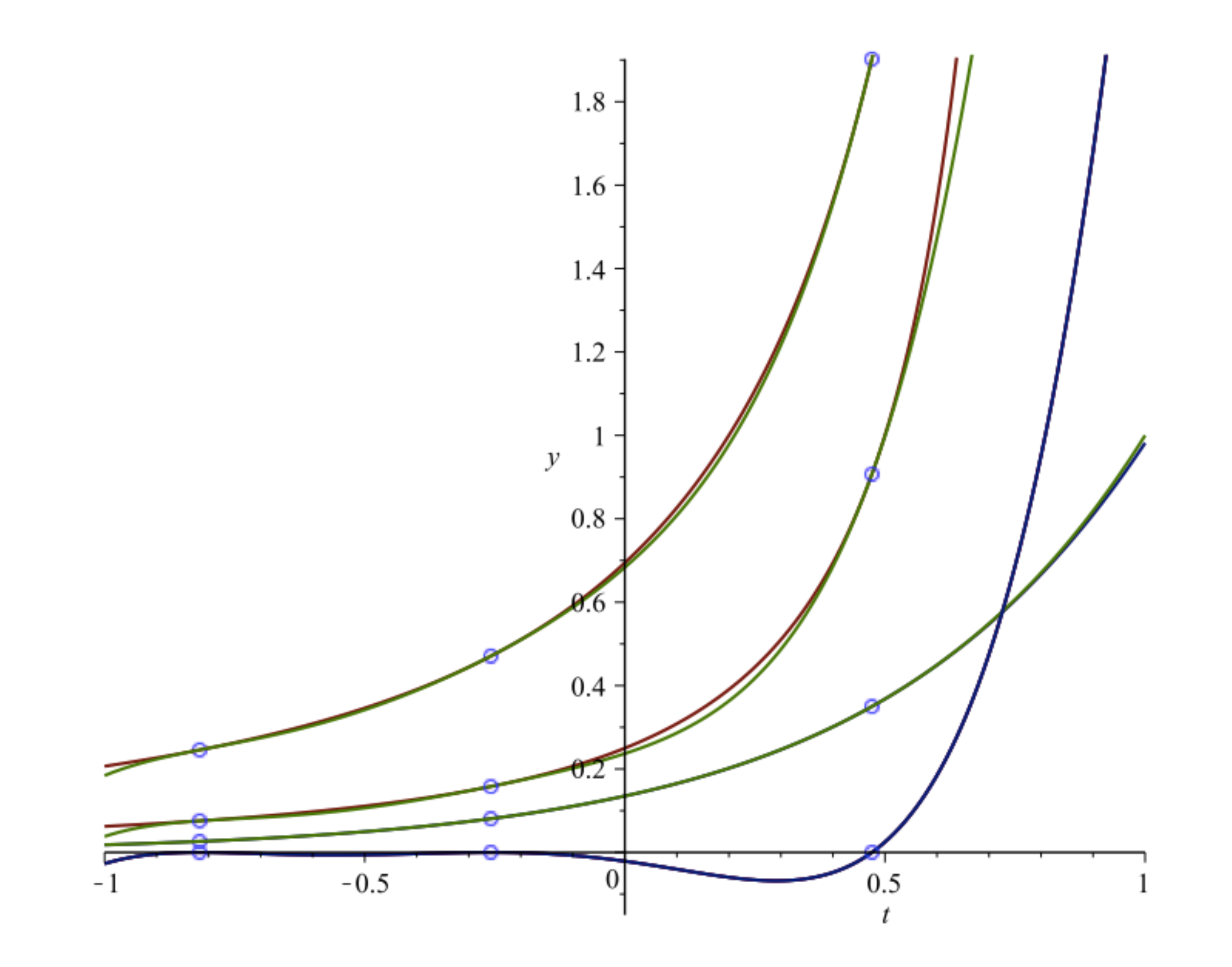}
\end{center}
\caption{The optimal polynomials (Hermite interpolants), that provide the ULB for Gauss, Korevaar, and Newton potentials (in ascending order), along with the corresponding Levenshtein polynomial for $n=4$, $N=24$}
\end{figure}

\subsection{Discussion and examples}
The bounds (\ref{bound_odd}) are easy for computation and
investigation. Moreover, the approach by which they were derived doesn't depend on the potential function and in this sense they are universal. This universality is illustrated in Figure 2, where we consider  $n=4$, $N=24$ and plot the Gauss, Korevaar, and Newton potential functions, together with the corresponding optimal Hermite interpolants of degree $\tau=5$, that solve the linear program \eqref{lp_problem} in the class $\mathcal{P}_5$. We also overlay the Levenshtein polynomial $f_5^{(4,s)}(t)$, whose zeros are the solutions of \eqref{op_eq}, where $s$ satisfies $L_5 (4,s)=24$. These zeros of the Levenshtein polynomial also serve as quadrature nodes for the universal lower bound \eqref{bound_odd} and as Hermite interpolation nodes for the optimal LP polynomials.

In \cite{BBCGKS} the authors have done an extensive experimental investigation of energy-minimizing point configurations, in particular they provide the computational minimizers for the Newton potential energy ($h(t)=[2(1-t)]^{-(n-2)/2}$) when $n=1,2,\dots,32$ and $N=1,2,\dots,64$. Table 1 compares the Newton energy from \cite{BBCGKS} and our universal lower bound (ULB) when $n=4$ and $N=5,6,\dots,64$.

\begin{table}[ht]\label{NEtable}
\begin{center}
\includegraphics[scale=.43]{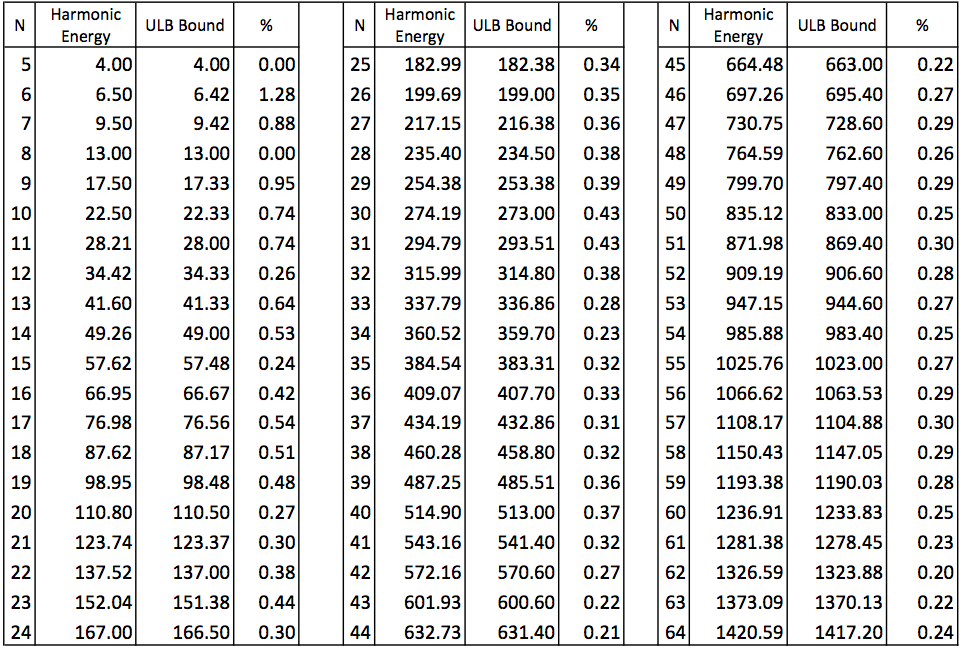}
\end{center}
\vskip 5mm
\caption{Newtonian (harmonic) energy comparison (see \cite{BBCGKS}) with ULB for $n=4$, $N=5, \dots ,64$.}
\end{table}

Utilizing the same Newton energy-minimizing configurations provided in \cite{BBCGKS} in Table 2 we compare  our universal lower bound (ULB) with the Gauss potential ($h(t)=e^{2t-2}$) energies of these configurations, which in general provide upper bounds on the minimal Gauss energy for the same choice of $n=4$ and $N=5,6,\dots,64$. We note that the error dramatically improves, which is to be expected, as the Hermite interpolants of analytic potential functions are excellent approximants. Observe that for $N=5$ and $N=8$ the bounds are exact. Both cases are universally optimal.

\begin{table}[ht]\label{GEtable}
\begin{center}
\includegraphics[scale=.43]{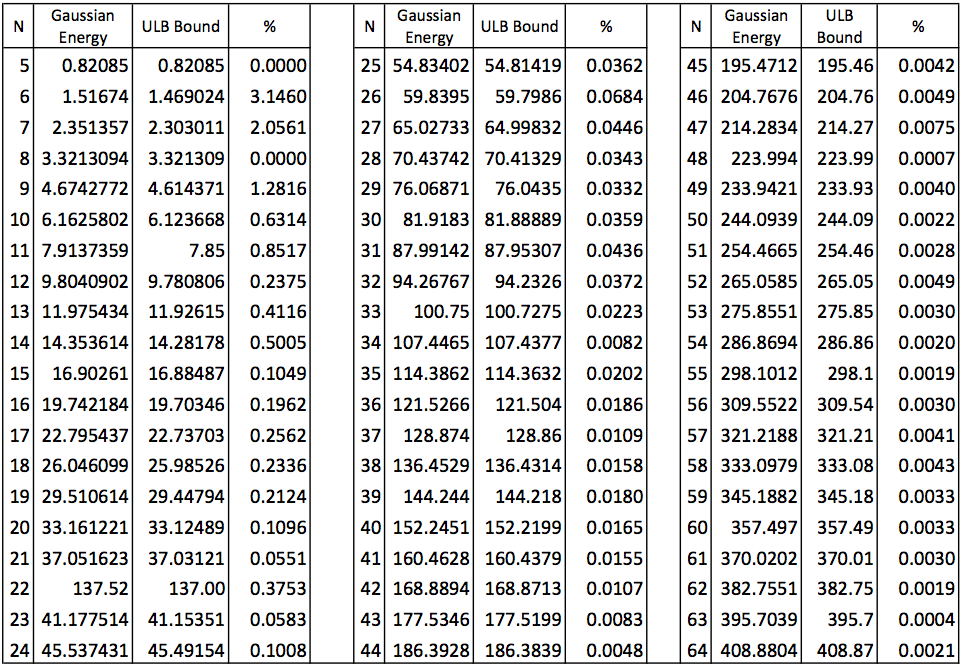}
\end{center}
\vskip 5mm
\caption{Gauss energy of the harmonic optimal configurations (as provided in \cite{BBCGKS}) compared with ULB for $n=4$, $N=5,\dots , 64$.}
\end{table}

As a consequence of the proof of Theorem \ref{thm3.2} we describe the explicit LP solutions for $m\leq \tau(n,N)$ in the next corollary.

\begin{corollary}
The linear program (LP) can be solved for any $m\le \tau(n,N)$ and the solution in the class $\mathcal{P}_m \cap A_{n,h}$ is given by the Hermite interpolants at the Levenshtein nodes determined by $N=L_m (n,s)$.
\end{corollary}

\begin{example}  Here we present the suboptimal LP solutions for $n=4$ and $N=24$.
In this case $\tau(n,N)=5$. For $m=1,\dots,5$ we find the intersection of $N=24$
with $L_1(4,s),\dots, L_5 (4,s)$. The corresponding suboptimal solutions as
Gegenbauer expansions (up to three digits) are:
\begin{eqnarray*}
f_1 (t)&=&.499P_0(t)+.229P_1(t)\\
f_2 (t)& =& .581P_0(t)+.305P_1(t)+0.093P_2(t)\\
f_3(t)&=&.658P_0(t)+.395P_1(t)+.183P_2(t)+0.069P_3(t)\\
f_4(t)&=&.69P_0(t)+.43P_1(t)+.23P_2(t)+.10P_3(t)+0.027P_4(t)\\
f_5(t)&=&.71P_0(t)+.46P_1(t)+.26P_2(t)+.13P_3(t)+0.05P_4(t)+0.01P_5(t).
\end{eqnarray*}
\end{example}

\vspace{-3mm}
\begin{figure}[ht]
\begin{center}
\vspace{1mm}
\includegraphics[scale=.37]{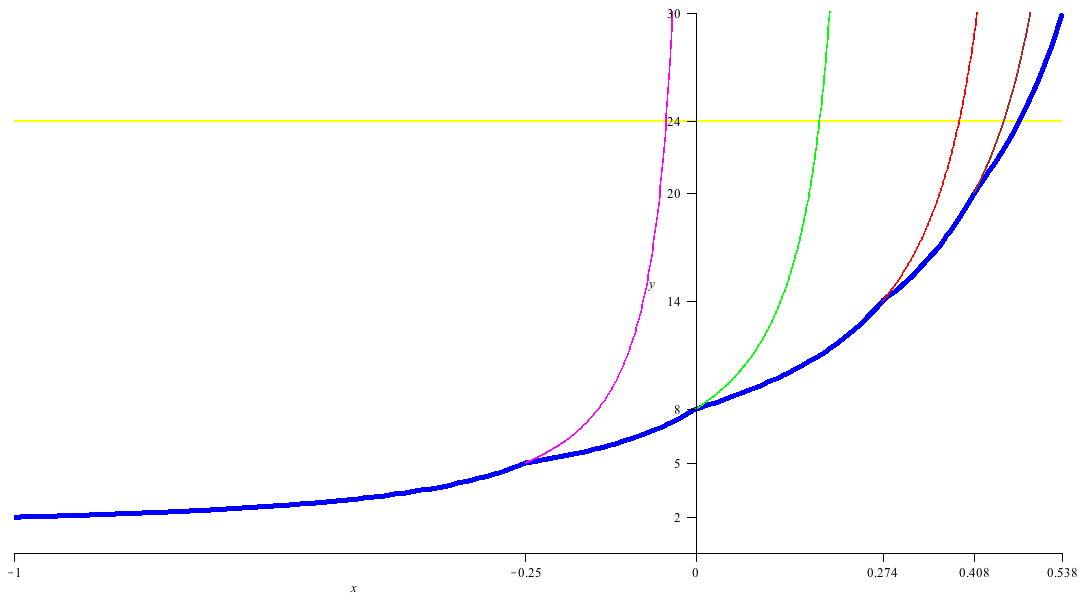}
\end{center}
\vskip 5mm
\caption{Suboptimal LP solutions for $n=4$ and $N=24$.}
\end{figure}

\smallskip

A natural question is whether linear programming bounds can be improved if we
consider polynomials of higher than $\tau(n,N)$ degree. The next section
investigates this topic. As one would expect from our results thus far presented,
the analogy with the situation for maximal spherical codes is quite close.

\section{Necessary and sufficient conditions for optimality of the universal lower bounds}

\subsection{Test functions}

Let $n$ and $N$ be fixed, $\tau=\tau(n,N)$ and $L_\tau(n,s)=N$ be as in \eqref{tauNn} and \eqref{Lns_eq}, and $j$ be a positive integer.
We introduce the following functions in $n$ and $s=\alpha_k$:
\begin{equation}
\label{test-functions}
Q_j(n,s):=\frac{1}{N}+\sum_{i=1}^{k} \rho_i P_j^{(n)}(\alpha_i)
                 \quad \text{for \ $s \in {\mathcal I}_{\tau}$}.
\end{equation}
It follows that $Q_j(n,s)=0$ for $1 \leq j \leq \tau$ and
$s \in {\mathcal I}_{\tau}$ (since this is the coefficient
$f_0=0$ in the Gegenbauer expansion of $P_j^{(n)}(t)$). Thus the functions $Q_j(n,s)$ are not interesting for these
cases and so we assume below that $j \geq \tau+1$ when $s \in {\mathcal I}_{\tau}$.

The next theorem shows that the functions $Q_j(n,s)$ give necessary and sufficient conditions for existence
of improving polynomials of higher degrees.

\medskip

\begin{theorem} \label{thm4.1}   The bounds (\ref{bound_odd})  can be improved by a
polynomial from $A_{n,h}$ of degree at least $\tau+1$ if and only if
$Q_j(n,s) < 0$ for some $j \geq \tau+1$. Furthermore, if $h$ is strictly absolutely monotone and $Q_j(n,s)<0$ for
some $j \geq \tau+1$, then (\ref{bound_odd}) can be improved by a polynomial
from $A_{n,h}$ of degree exactly $j$.
\end{theorem}

\begin{proof} We give a proof for $\tau=2k-1$.

(Necessity)  The necessity follows from Theorem \ref{THM_subspace_improve} for $I=\{2k,2k+1,\dots \}$.

\smallskip

(Sufficiency) Conversely, assume that $h$ is strictly absolutely monotone and suppose that $Q_j(n,s) <0$ for some $j \geq 2k$.

We shall improve the bound (\ref{bound_odd}) by using the polynomial
\[ f(t)=\epsilon P_j^{(n)}(t)+g(t), \]
where $\epsilon >0$ and $g(t)\in \mathcal{P}_{2k-1}$ will be properly chosen. Denote $\tilde{h}(t):= h(t)-\epsilon P_j^{(n)}(t)$ and select $\epsilon$ such that $\tilde{h}(t)^{(i)}(t) \geq 0$ on $[-1,1]$ for all $i=0,1,\dots,j$. Observe, that for this choice of $\epsilon$ the function $\tilde{h}(t)$ is absolutely monotone. The polynomial $g(t)$  is chosen as the Hermite interpolant of $\tilde{h}$ at the nodes $\{ \alpha_i\}$, i.e.
\[ g(\alpha_i)=\tilde{h}(\alpha_i), \ g^\prime(\alpha_i)=\tilde{h}^\prime(\alpha_i), \ i=1,2,\ldots,k. \]
Since $\tilde{h}(t)$ is an absolutely monotone function, we infer as in Theorem \ref{thm3.2} that $g \in A_{n,\tilde{h}}$, implying that $f\in A_{n,h}$.

Let $g(t)=\sum_{\ell=0}^{2k-1} g_\ell P_\ell^{(n)}(t)$. Note that $f_0=g_0$ and $f(1)=g(1)+\epsilon$.
We next prove that the bound given by $f(t)$ is better that $R_{2k-1}(n,N;h)$. To this end, we
multiply by $\rho_i$ and sum up the first interpolation equalities:
\[ \sum_{i=1}^{k} \rho_i g(\alpha_i)= \sum_{i=1}^{k} \rho_i h(\alpha_i)-\epsilon \sum_{i=1}^{k} \rho_i P_j^{(n)}(\alpha_i). \]
Since $$\sum_{i=1}^{k} \rho_i g(\alpha_i)=g_0-\frac{g(1)}{N}$$ by (\ref{defin_qf}) and
$$\sum_{i=1}^{k} \rho_i P_j^{(n)}(\alpha_i)=Q_j(n,s)-\frac{1}{N} $$
by the definition of the test functions (\ref{test-functions}), we obtain
\[ g_0-\frac{g(1)}{N}=\frac{R_{2k-1}(n,N;h)}{N^2}+\frac{\epsilon}{N}-\epsilon Q_j(n,s) \]
which is equivalent to
\[ N(Ng_0-(g(1)+\epsilon))=R_{2k-1}(n,N;h)-\epsilon N^2Q_j(n,s). \]
Therefore $N(Nf_0-f(1))=R_{2k-1}(n,N;h)-\epsilon N^2Q_j(n,s)>R_{2k-1}(n,N;h)$, i.e. the polynomial $f(t)$ gives
better bound indeed. We also obtained a new bound
\begin{equation}
\label{imp_high_deg}
 W(n,N;h) \geq R_{2k-1}(n,N;h)-\epsilon N^2Q_j(n,s).
\end{equation}

\end{proof}

Theorem 4.1 provides a sufficient condition for solving the infinite linear program \eqref{lp_problem}.

\begin{corollary}
\label{LP-cor}
If $Q_j(n,s) \geq 0$ for all $j>\tau (n,N)$, then $f_{\tau(n,N)}^h (t)$ solves the linear program \eqref{lp_problem}.
\end{corollary}

\subsection{Investigation of the test functions}

The test functions (\ref{test-functions}) coincide with the functions with the same name which were
introduced and investigated in 1996 by Boyvalenkov, Danev and Bumova \cite{BDD}. More details and all proofs are given in
the dissertations \cite{BoumovaPhD} and \cite{Boy_dis}. We cite some results
from \cite{BDD,BoumovaPhD,Boy_dis} with only reformulations for energy bounds.

\begin{theorem}[\cite{BoumovaPhD}, \cite{Boy_dis},  \cite{BDD}] \label{thm5.1}
The bounds $R_\tau(n,N;h)$ cannot be improved by using polynomials of degrees $\tau +1$
and $\tau +2$.
\end{theorem}

Set $k_1(n):=\sqrt{n-2}$ and let $k_2(n) \geq 9$ be such that
\[ 4n \leq k_2(n)^2-4k_2(n)+5+\sqrt{k_2(n)^4-8k_2(n)^3-6k_2(n)^2+24k_2(n)+25}. \]
Then we have the following theorems.

\begin{theorem} \label{thm5.2}    a) {\rm \cite[Theorem 3.5.15]{BoumovaPhD}, \cite[Theorem 3.4.12]{Boy_dis}}
If $n \geq 3$ and $k \geq k_1(n)$,
then all bounds $R_{2k}(n,N;h)$ corresponding to $s$ in the open interval ${\mathcal I}_{2k}$
can be improved by polynomials of degree $2k+3$.

b) {\rm  \cite[Theorem 3.5.9]{BoumovaPhD}, \cite[Theorem 3.4.14]{Boy_dis}} If $n \geq 3$ and $k \geq k_2(n)$,
then all bounds $R_{2k-1}(n,N;h)$ corresponding to $s$ in the open interval ${\mathcal I}_{2k-1}$
can be improved by polynomials of degree $2k+3$.
\end{theorem}

\begin{theorem} \label{thm4.5}    a) If $n \geq 3$ and $k \geq k_1(n)$, then
\begin{equation}
\label{imp_odd_high_deg_m+4}
\mathcal{E} (n,N;h) \geq R_{2k-1}(n,N;h)-\epsilon N^2Q_{2k+3}(n,s).
\end{equation}
for every $N \in (D(n,2k-1),D(n,2k))$ where $\epsilon$ is chosen as in Theorem \ref{thm4.1}.

b) If $n \geq 3$ and $k \geq k_2(n)$, then
\begin{equation}
\label{imp_even_high_deg_m+3}
\mathcal{E} (n,N;h) \geq R_{2k}(n,N;h)-\epsilon N^2Q_{2k+3}(n,s).
\end{equation}
for every $N \in (D(n,2k),D(n,2k+1))$ where $\epsilon$ is chosen as in Theorem \ref{thm4.1}.
\end{theorem}

{\it Proof.} This follows from (\ref{imp_high_deg}) and the fact that Theorem \ref{thm5.2} is based on the inequality
$Q_{2k+3}(n,s)<0$ which holds true for the mentioned values of $n$ and $\tau$.  \hfill $\Box$

\medskip

Another application of Theorem \ref{thm5.2}  concerns the sharp configurations.
Recall that a sharp configuration is a maximal spherical
$(n,L_{2k-1}(n,s),s)$ =(di\-men\-sion, cardinality, maximal cosine) code; i.e. a code that
attains the odd Levenshtein bound $L_{2k-1}(n,s)$ (cf. \cite{Lev3}).
In fact, the next corollary is implicit in \cite{BDL} and follows from the main result of \cite{MW} as well.

\begin{corollary}
 For any fixed dimension $n \geq 3$ only finitely many sharp configurations are possible.
 \end{corollary}

{\it Proof.} Theorem  \ref{thm5.2} implies that in every fixed dimension
$n \geq 3$ every Levenshtein bound $L_{2k-1}(n,s)$
can be improved in the whole open interval $\left( t_k^{1,0}, t_k^{1,1} \right)$
provided $k$ is large enough. The remaining end points correspond to
tight spherical designs, which means (among many other things) that $k \leq 6$ \cite{BD1,BD2}.
This leaves only finitely many possible intervals $\mathcal{I}_{2k-1}$
where the Levenshtein bound $L_{2k-1}(n,s)$ can be attained. Every such interval contains finitely many $s$,
corresponding to cardinalities $N$, which completes the proof. \hfill $\Box$

We complete the subsection with the following conjecture,
based on the above results and numerous investigations of the test functions as
related to maximal spherical codes.

\begin{conjecture}
If $Q_j (n,s)\geq 0$ for $j=\tau (n,N)+3$ and $\tau (n,N)+4$, then $Q_j (n,s)\geq 0$ for all $j>\tau (n,N)$.
\end{conjecture}

\subsection{Test functions and LP universality}

We now apply the test functions to the study of universal configurations.

\begin{defn} A spherical code $C\subset \mathbb{S}^{n-1}$ of cardinality $|C|=N$ is called  {\it LP-universally optimal} if
\[ E(C;h)=\mathcal{W}(n,N,\mathcal{P};h), \quad {\rm \ for\ all\ absolutely\ monotone} \quad h,\]
where $\mathcal{P}$ is the subspace of polynomials.
\end{defn}

\begin{remark}
Observe that from \eqref{EKlowbound} and \eqref{Wdef} one infers that LP-universally optimal codes are in fact universally optimal. If the conjecture in Ballinger et al \cite{BBCGKS} is true, then Theorem \ref{LastThm} implies that the converse does not hold.
\end{remark}

We derive a criterion for positivity of test functions of large enough $j$ that can be used for proving that certain spherical codes of given dimension $n$ and cardinality $N$ are not LP-universally optimal. We utilize $(n,N)$ to denote\footnote{We note that \cite{BBCGKS} uses $(N,n)$ notation instead.} codes $C \subset \mathbb{R}^n$ with cardinality $| C |=N$.
As examples, the cases $(n,N)=(10,40)$, $(14,64)$ and $(15,128)$ are analyzed.

Sharp estimations for Gegenbauer polynomials can be derived from \cite{EMN} (see also
\cite{Kra}). In \cite[Theorem 1]{EMN} the following inequality is given
\begin{equation}
\label{Ineq1}
\max_{t\in[-1,1]}\sqrt{1-t^2}w(t)p_j^2 (t)\leq \frac{2e(2+\sqrt{\alpha^2+\beta^2})}{\pi},
\end{equation}
where $\{ p_j (t)\}$ are the orthonormal Jacobi polynomials with weight $w(t)=(1-t)^\alpha(1+t)^\beta$. Utilizing $\alpha=\beta=\frac{n-3}{2}$ to get Gegenbauer polynomials and the normalization $P_j^{(n)}(1)=1$, we rewrite \eqref{Ineq1} as
\begin{equation}
\label{Ineq3}
|P_j^{(n)} (t)|
\leq \frac{\Gamma\left( \frac{n-1}{2} \right) }{(1-t^2)^{(n-2)/4}}
\sqrt{\frac{2^{n-2} e(4+(n-3)\sqrt{2})\, j!}{\pi(2j+n-2)\, (j+n-3)!}},
\end{equation}
where $\Gamma(x)$ is the Gamma function \cite{Wat}. Note that for every fixed $n\geq 3$ and $t\in(-1,1)$ the right-hand side of \eqref{Ineq3} is strictly monotone decreasing in $j$.

Let $k$, $\alpha_1$, $\alpha_2$ and $\rho_1$ be as in Theorem~\ref{thm3.2}. Denote by
$j_0(n,N)$ the smallest degree $j>\tau(n,N)$ such that the right hand side of
\eqref{Ineq3} is less than $\frac{1}{N-1}$ when
\begin{equation}
\label{def_t_alpha}
t=\begin{cases} \alpha_1 &  \text{if $\alpha_1>-1$},\\
\alpha_2 & \text{if $\alpha_1=-1$ and $\rho_1<\frac{1}{N}$},
\end{cases}
\end{equation}
or less than $\frac{2}{N-2}$ when $t=\alpha_2$ if $\alpha_1=-1$ and $\rho_1=1/N$.

\begin{theorem}
\label{Qposthm}
Let $n\ge 3$, $N\ge 2$, and let $k$, $\alpha_1$, $\alpha_2$ and $\rho_1$ be as in Theorem~\ref{thm3.2}.
Then $Q_j(n,\alpha_k)\ge 0$ for all $j\ge j_0(n,N)$.
\end{theorem}

\begin{proof}
As the comments on the dynamical
behavior of the quadrature nodes $\{\alpha_i\} $ at the end of Section 2 indicate, we have $|\alpha_1| \ge |\alpha_i|$ for $i=2, \ldots, k$ and in the case
$\alpha_1=-1$, we further have $|\alpha_2|\ge |\alpha_i|$ for $i=3, \ldots, k$.

We first consider the case $|\alpha_1|<1$; i.e., $\alpha_1>-1$. If $j\ge j_0(n,N)$
then we have
\begin{equation}\label{Qposbnd}
Q_j(n,s) \geq \frac{1}{N}-\sum_{i=1}^k\rho_i |P_j^{(n)}(\alpha_i)| \ge \frac{1}{N}-\left(1-\frac{1}{N}\right)\cdot \frac{1}{N-1}=0
\end{equation}
(we used $N\sum_{i=1}^k \rho_i=N-1$ following from \eqref{defin_f0.1} for $f(t)=1$).
The case $\alpha_1=-1$ and $\rho_1<1/N$ is handled similarly
using \eqref{Ineq3} as suggested by the second line of \eqref{def_t_alpha}.

For the final special case $\alpha_1=-1$ and $\rho_1=1/N$ it is clear
(cf. \cite{BDL}) that $Q_j(n,s)=0$ for odd $j$. The case of even $j$ follows similarly
as above using the facts that $P_j^{(n)}(-1)=1$ and that $|\alpha_i|\le |\alpha_2|$ for $i= 3,\ldots, k$.
\end{proof}

Theorem \ref{Qposthm} gives a useful tool for disproving LP-universal optimality. For given $n$ and $N$ and
numerics suggesting that Corollary \ref{LP-cor} may hold one finds explicit $j_0(n,N)$ and calculates the
remaining test functions $Q_j(n,s)$ for every $j \in \{\tau(n,N)+3,\tau(n,N)+4,\ldots,j_0(n,N)-1\}$.
This will be applied in the next subsection for some codes from \cite{BBCGKS}.

\subsection{Examples}

Table 3 lists the first twenty test functions for some interesting configurations.
We utilize $(n,N)$ to denote codes $C \subset \mathbb{R}^n$ with cardinality $| C |=N$.

\begin{table}[ht]
\begin{center}
\vspace{1mm}
\includegraphics[scale=.52]{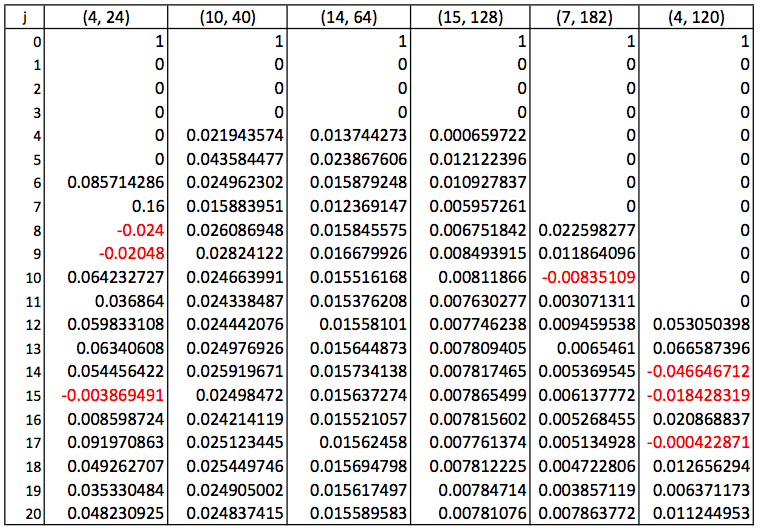} \\
\caption{Test functions for some special $(n,N)$ spherical codes.}
\end{center}
\end{table}

Judging by the behavior of the test functions the linear programming
method will provide improvements on our ULB for $(4,24)$ and $(7,182)$
but it is unlikely to give a solution similar to the case with the $600$-cell
$(4,120)$, where a polynomial in $\mathcal{P}_{17}$ served as an exact lower bound.
Indeed, that the test functions $Q_{14}, Q_{15}$ and $Q_{17}$ are negative provides
additional insight on the unique property of the $600$-cell as the only universally
optimal code known that is not a sharp configuration.

The first configuration $(4,24)$ is the $D_4$ root system, or the so-called
kissing number configuration in $\mathbb{R}^4$ (see \cite{M}),
which was shown by Cohn, Conway, Elkies, and Kumar (see \cite{CCEK})
not to be universal. The negative test functions $Q_8(4,s)$ and
$Q_9(4,s)$, $s=\alpha_2 \approx 0.4749504897$, suggest searching for a polynomial
$f(t)=\sum_{i=0}^9 f_i P_i^{(4)}(t)$ with $f_6=f_7=0$ and four touching
points of the graphs of $f(t)$ and the potential $h(t)$. We have developed a numerical algorithm for handling such situations. For example, if $h(t)=\frac{1}{2(1-t)}$ is the Newton potential,
our numerical calculations led to the polynomial
\begin{eqnarray*}
f(t) &=& 0.4987 + 0.4852t + 0.4535t^2 + 0.5546t^3 + 0.9401t^4 + 0.8425t^5 \\
&& -\,0.3305t^6  - 0.7479t^7 + 0.1889t^8 + 0.37394t^9 \\
&=& 0.0073P_9^{(4)}(t)+ 0.0066P_8^{(4)}(t)+0.0659P_5^{(4)}(t)+ 0.2384P_4^{(4)}(t)+ 0.5116P_3^{(4)}(t) \\
&& +\,0.7915P_2^{(4)}(t)+ 0.9236P_1^{(4)}(t)+ 0.7142P_0^{(4)}(t).
\end{eqnarray*}
The Hermite interpolation points are approximately $-0.860297$, $-0.489872$,
$-0.195724$ and $0.47850$. The bound obtained from $f(t)$
is $333.1575$, while the universal lower bound (\ref{bound_odd}) gives
$R_5(4,24;1/(2(1-t))=333$ and the energy of the $D_4$ root system is 334.
Theoretical and computational aspects of the aforementioned algorithm for improvements (when possible) of our ULB and their nature will be discussed elsewhere.

\begin{theorem}\label{LastThm}
The spherical codes $(n,N)=(10,40)$, $(14,64)$ and $(15,128)$ are not LP-universally optimal.

\end{theorem}

\begin{proof}
The codes $(10,40)$ and $(14,64)$ were conjectured by
Ballinger, Blekherman, Cohn, Giansiracusa, Kelly, and Sch\"{u}rmann in \cite{BBCGKS}
to be universally optimal. It follows from Theorem \ref{Qposthm} and numerical calculations
as explained in the end of the last subsection that these codes are not LP-universally optimal.
Indeed, we have $\tau(10,40)=3$ (so $\alpha_1>-1$), $j_0(10,40)=10$ and the second column in
Table 3 shows that this code is not LP-universally optimal. Similarly, $\tau(14,64)=3$, $j_0(14,64)=8$,
and the inspection of the third column of Table 3 suffices. The code $(15,128)$ was not conjectured
to be universally optimal (but not eliminated) in \cite{BBCGKS} and we see that it is not
LP-universally optimal because of $\tau(15,128)=3$, $j_0(14,64)=9$, and the fourth column in Table 3.
\end{proof}

\end{document}